\documentclass[12pt]{article}
\usepackage{amssymb}
\usepackage{amsmath}
\usepackage{amsthm}
\usepackage{graphicx}
\usepackage{graphics}
\usepackage{amsthm}

\input pictex.tex
\newcommand{\esp}{\hspace{0.05cm}}

\theoremstyle{definition}

\newtheorem{thm}{Theorem}[section]

\newtheorem{prop}[thm]{Proposition}

\newtheorem{lem}[thm]{Lemma}

\newtheorem{rem}[thm]{Remark}

\begin{document}

\date{}
\author{Andr\'es Navas}

\title{An $L^1$ ergodic theorem with values in a nonpositively curved space 
via a canonical barycenter map}
\maketitle

\vspace{-0.15cm}

The extension of classical ergodic theorems to a geometric --nonpositively curved-- 
setting has been one of the most fascinating developments of Ergodic Theory over 
the last years; see \cite{KL} for a nice survey containing most of the relevant 
results for functions (cocycles) taking values in isometry groups. 

In a different though related direction, A.~Es-Sahib and H.~Heinich proved in 
\cite{EH} an ergodic type theorem for $L^1$ i.i.d. random variables taking values 
in a nonpositively curved space. An analogous result for $L^2$ i.i.d. random variables  
was given by K.-T.~Sturm in \cite{sturm}. Recently, T.~Austin proved a 
nice extension of Sturm's result to arbitrary measure-preserving actions of amenable 
groups (see \cite{austin}). Unfortunately, Austin's 
$L^2$-setting is not the most appropriate one in view of that the most powerful framework 
of the ergodic theorem is that of $L^1$ spaces. In this work, we prove a general ergodic 
theorem for $L^1$ functions taking values in nonpositively curved spaces, where the 
notion of Birkhoff sums is replaced by that of barycenters along the orbits.

Let us begin by recalling a classical construction. Given a complete CAT(0)-space $(X,d)$, 
we consider the space $P^2(X)$ of probability measures with finite second moment, that is, 
$$\int_{X} d(x,y)^2  \esp d \mu(y) < \infty$$
(this condition does not depend on the point $x \!\in\! X$). 
Following Cartan (see for instance \cite{jost}), to each $\mu \!\in\! P^2 (X)$ one may associate 
a {\em barycenter} $bar(\mu)$, namely the unique point that minimizes the function 
$$x \to \int_{X} d(x,y)^2 \esp d\mu(y).$$
A crucial property of $bar \!\!: P^2 (X) \!\to\! X$ is that 
it is 1-Lipschitz for the $2$-Wasserstein metric \cite{sturm}:   
$$d \big( bar (\mu_1),bar ({\mu_2}) \big) \leq W_2 (\mu_1,\mu_2) 
:= \inf_{\nu \in (\mu_1|\mu_2)} \sqrt{\int_{X \times X} d(x,y)^2 \esp d \nu (x,y)},$$
where $(\mu_1|\mu_2)$ denotes the set of all probability measures $\nu$ on $X \times X$ 
that project into $\mu_1$ and $\mu_2$ on the first and the second factor, respectively 
(see \cite{villani} for more details on this metric).

The first task of this work was to introduce an analogous notion for 
the space $P^1(X)$ of probability measures with finite first moment:
$$\int_{X} d(x,y)  \esp d \mu(y) < \infty.$$
It was after we developed a notion of barycenter adapted to our needs that we 
discovered the equivalent construction of \cite{EH}. We decided to include our approach 
here because it is more elementary in that, unlike \cite{EH}, it does not rely on 
deep probabilistic results. Although this makes our computations a little bit 
more involved, it has the advantage of allowing us to avoid the (finite) local 
compactness hypothesis of \cite{EH} for the underling space, thus solving a problem 
formulated in \cite[Example 6.5]{sturm}. Summarizing, let $(X,d)$ be a complete 
metric space with nonpositive curvature in the sense of Buseman 
(a {\em Buseman space}, for short). Assuming that $X$ is separable, 
in \S \ref{baricentro} we construct a map $bar^{\star}\! : P^1 (X) \to X$ 
that is 1-Lipschitz for the 1-Wasserstein metric:
$$d \big( bar^{\star} (\mu_1), bar^{\star} ({\mu_2}) \big) \leq W_1 (\mu_1,\mu_2) 
:= \inf_{\nu \in (\mu_1|\mu_2)} \int_{X \times X} d(x,y) \esp d \nu (x,y).$$ 
By elementary reasons, this also applies to any separable Banach space, where 
geodesic are understood as being segments of lines. 

The map constructed above is equivariant with respect to the natural action of 
isometries. At the end of \S \ref{baricentro}, we give an application of this fact,   
namely, we prove that every compact group of isometries of a Buseman space has 
a fixed point. The novelty here is that we do not assume any hypothesis of 
strict convexity (having such an hypothesis, the result is elementary and 
well-known).  

\vspace{0.2cm}

We next enter into the goal of this work. Given an amenable group $G$ with 
a measure-preserving action $T$ on a probability space $(\Omega, \mathcal{P})$, 
let $(F_n)$ be a {\em tempered} F\o lner sequence in $G$, that is, a F\o lner 
sequence for which there exists $C > 0$ such that for all $n \in \mathbb{N}$, 
$$m_G \Big( \bigcup_{k < n} F_k^{-1} F_n \Big) \leq C m_G(F_n),$$
where $m_G$ denotes the left Haar measure on $X$. Let $\varphi: \Omega \to X$ 
be a measurable function lying in $L^1 (\mathcal{P},X)$, that is, such that 
for some (equivalently, all) $x \in X$,
$$\int_{\Omega} d \big( \varphi(\omega),x \big) \esp d \mathcal{P} (\omega) < \infty.$$ 
Notice that $L^1 (\mathcal{P},X)$ becomes a metric space when endowed with the distance
$$d_1 (\varphi,\psi) := \sqrt{\int_{\Omega} 
d \big(\varphi(\omega),\psi(\omega) \big) \esp d \mathcal{P} (\omega)}.$$

\vspace{0.1cm}

\noindent{\bf Main Theorem.} {\em With the notation above, assume that $X$ is either 
a separable Banach space or a separable Buseman space. Then}
$$\omega \mapsto bar^{\star} \left( \frac{1}{m_G (F_n)} \int_{F_n} 
\delta_{\varphi(T^{^g} \omega)} \esp d m_G (g)\right)$$
{\em is a sequence of maps that converges pointwise and in $L^1 (\mathcal{P},X)$ 
to a $T$-invariant function from $\Omega$ to $X$.}

\vspace{0.5cm}

For Banach spaces, the barycenter of a measure \esp 
$\frac{1}{m} \big( \delta_{x_1} + \cdots + \delta_{x_m} \big)$ \esp 
is just the Dirac measure concentrated at the point 
\esp $\frac{1}{m} (x_1 + \cdots + x_m)$. \esp In particular, 
when $G \sim \mathbb{Z}$, $X = \mathbb{R}$ and $F_n = \{0,\ldots,n-1\}$, 
the theorem reduces to the classical (invertible) Birkhoff ergodic 
theorem for $\varphi \in L^1 (\mathcal{P},\mathbb{R})$.

The proof of the Main Theorem uses the general strategy of \cite{austin}, 
namely the contractivity properties 
of the barycenter maps transforms the desired convergence into that of suitable 
sequences of real-valued functions to which Lindenstrauss' pointwise ergodic 
theorem \cite{lindenstrauss} applies. Recall that in the setting of \cite{austin}, 
the probability measure lies in $P^2 (X)$ and one considers functions 
$\varphi: \Omega \to X$ lying in the space $L^2 (\mathcal{P},X)$, 
that is, such that for some (equivalently, all) $x \in X$,
$$\int_{\Omega} d \big( \varphi(\omega),x \big)^2 
\esp d \mathcal{P}(\omega) < \infty.$$ 
This space may be naturally endowed with the distance
$$d_2 (\varphi,\psi) := 
\int_{\Omega} d \big( \varphi(\omega),\psi(\omega) \big)^2 
\esp d \mathcal{P}(\omega).$$
Austin's theorem then asserts that for every \esp 
$\varphi \in L^2 (\mathcal{P},X)$, \esp the sequence of maps
\begin{equation}\label{austin-stat}
\omega \mapsto bar \left( \frac{1}{m_G (F_n)} \int_{F_n} 
\delta_{\varphi(T^{^g} \omega)} \esp d m_G (g)\right)
\end{equation}
converges pointwise and in $L^2 (\mathcal{P},X)$ to 
a $T$-invariant function from $\Omega$ to $X$. 

Quite interestingly, Austin's theorem is not a consequence of our Main Theorem. Indeed, 
although --as in the classical case-- our theorem extends to an $L^p$-version  by a 
straightforward and well-known argument, the barycenters $bar$ and $bar^{\star}$ 
may differ, even for very nice spaces; see Remark \ref{raro}. Despite of this, 
the map $bar$ is also 1-Lipschitz for the 1-Wasserstein metric; see 
\cite[Proposition 4.3]{sturm}. Using the methods of \S \ref{proof}, 
this allows showing that the convergence of the sequence of maps 
(\ref{austin-stat}) actually holds in $L^1 (\mathcal{P},X)$. We point out 
that this still holds for probability measures in $P^1(X)$ for a clever 
modification of Cartan's barycenter (see \cite[Proposition 4.3]{sturm}). 

\vspace{0.35cm}

\begin{small}

\noindent{\bf Acknowledgments.} It is a pleasure to thank A.~Karlsson for useful hints 
and references concerning fixed points for actions on Buseman spaces, J.~Bochi for 
inspiring discussions on the barycenter map, and K.-T.~Sturm for a clever remark. 

This work 
was funded by a Fondecyt Research Project and the Math-AMSUD Research Project DySET. 
\end{small}


\section{The barycenter map}
\label{baricentro}

\hspace{0.5cm} For a Banach space $X$, a natural definition of barycenter 
of a measure $\mu \in P^1 (X)$ is 
$$bar^{\star} (\mu) := \int_X x \esp d\mu(x).$$
Notice that given $\mu_1,\mu_2$ in $P^1(X)$, for each $\nu \in (\mu_1|\mu_2)$ 
we have\\

\begin{eqnarray*}
\int_{X \times X} \|x - y\| \esp d\nu(x,y) 
&\geq& 
\left\| \int_{X \times X} (x-y) \esp d\nu (x,y) \right\|\\ 
&=& 
\left\| \int_{X \times X} x \esp d\nu (x,y) - \int_{X \times X} y \esp d\nu (x,y) \right\|\\ 
&=&
\left\| \int_{X} x \esp d (\pi_{\!{_1}} \! \nu) (x) - 
\int_{X} y \esp d (\pi_{\!{_2}} \! \nu) (y) \right\|\\
&=& 
\big\| bar^{\star} (\mu_1) - bar^{\star} (\mu_2) \big\|.
\end{eqnarray*}
As a consequence, 
$$\big\| bar^{\star} (\mu_1)- bar^{\star} (\mu_2) \big\| 
\leq W_1 (\mu_1,\mu_2).$$

A definition with an analogous property for nonpositively curved spaces is much 
more subtle. In what follows, $X$ will denote a Buseman space (separability 
will be needed later). Recall 
that this means that $X$ is geodesic and the distance function along geodesics is 
convex. Equivalently, given any two pairs of points $x,y$ and $x',y'$, their 
corresponding (unique) midpoints $m,m'$ satisfy
\begin{equation}\label{media}
d(m,m') \leq \frac{d(x,x')}{2} + \frac{d(y,y')}{2}.
\end{equation}

This property allows defining a barycenter $bar_n (x_1,\ldots,x_n )$ 
of any finite family $(x_1,\ldots,x_n)$ of (nonnecessarily distinct) 
points as follows. For $n = 1$, we let $bar_1 (x) := x$. For $n \!=\! 2$, 
we let $bar_2 (x,y)$ be the midpoint between $x$ and $y$. Now, assuming that 
the barycenters \esp $bar_n (\cdot,\ldots,\cdot)$ \esp of all families of $n$ points 
have been defined, we define $bar_{n+1} (x_1, \ldots, x_n, x_{n+1} )$ as follows: 
Starting with $(x_1,\ldots,x_{n+1})  =: (x_1^{(0)}, \ldots, x_{n+1}^{(0)} )$,
we replace each $x_i$ by the (already defined) barycenter of 
$(x_1, \ldots, x_{i-1}, x_{i+1}, \ldots, x_{n+1} )$. Then we do the same
with the resulting set $\{x_1^{(1)}, \ldots, x_{n+1}^{(1)} )$, thus yielding
a new set $(x_1^{(2)},\ldots,x_{n+1}^{(2)}\}$. Repeating this procedure and
passing to the limit along the Cauchy sequences $(x_i^{(k)})_{k \in \mathbb{N}}$, 
the corresponding set with collapse to a single point, that we call the 
barycenter of $(x_1,\ldots,x_{n+1})$. The proof of this convergence 
will be accomplished inductively together with the following
crucial relation:
\begin{equation}\label{needed-estimate}
d \big( bar_n (x_1,\ldots,x_n) , bar_n (x_1,\ldots,x_n) \big)
\leq \frac{1}{n} \sum_{i=1}^{n} d(x_i, y_i).
\end{equation}

First, for $n=2$, the barycenter is already defined, and (\ref{needed-estimate}) reduces 
to (\ref{media}). Now, assuming that we have showed the existence of the barycenter as 
well as inequality $(\ref{needed-estimate})$ for families of $n$ points, let us consider
a family $(x_1,\ldots,x_{n+1})$. For each $i \neq j$ in $\{ 1,\ldots,n+1 \}$, we have
$$d(x_i^{(1)}\!, x_j^{(1)}) = d \big( bar_{n} (x_1,...,x_{i-1},x_{i+1},...,x_{n+1}), 
bar_n (x_1,...,x_{j-1},x_{j+1},...,x_{n+1}) \big)
\leq \frac{d(x_i, x_j)}{n}.$$
Therefore,
$$diam \{ x_1^{(1)}, \ldots, x_{n+1}^{(1)} \} 
\leq \frac{1}{n} \esp diam \{x_1, \ldots, x_{n+1}\},$$
and more generally, for all $k \geq 1$,
$$diam \{ x_1^{(k)}, \ldots, x_{n+1}^{(k)} \} \leq
\frac{1}{n^k} \esp diam \{x_1, \ldots, x_{n+1}\}.$$
By this inequality and Lemma \ref{one} below, the diameter of the convex closure of 
$\{x_1^{(k)}, \ldots, x_{n+1}^{(k)}\}$ converges to zero as $k$ goes to infinite. 
Since $x_i^{(l)}$ belongs to this convex closure for all $l\geq k$, this shows 
that \esp $bar_{n+1} (x_1, \ldots, x_{n+1})$ \esp is well defined.

Next, take two families $(x_1,\ldots,x_{n+1})$ and $(y_1,\ldots,y_{n+1})$.
By the inductive hypothesis, for each index $i \in \{1,\ldots,n+1\}$,
$$d(x_i^{(1)}\!\!, y_i^{(1)}) = d \big( bar_n (x_1,..., x_{i-1}, x_{i+1},..., x_{n+1}), 
bar_n (y_1,..., y_{i-1}, y_{i+1},..., y_{n+1}) \big) 
\leq \frac{1}{n} \! \sum_{j \neq i} \! d(x_j, y_j).$$
Summing over all $i=1,\ldots,n+1$, this yields
$$\sum_{i=1}^{n+1} d(x_i^{(1)}, y_i^{(1)}) \leq \sum_{i=1}^{n+1} d(x_i, y_i).$$
More generally, for all $k \geq 1$,
$$\sum_{i=1}^{n+1} d(x_i^{(k)}, y_i^{(k)}) \leq \sum_{i=1}^{n+1} d(x_i^{(k-1)}, y_i^{(k-1)}) 
\leq  \ldots \leq \sum_{i=1}^{n+1} d(x_i, y_i).$$
Letting $k$ go to infinite, all the points $x_i^{(k)}$ (resp. $y_i^{(k)}$) converge 
to $bar_{n+1} (x_1,\ldots,x_{n+1})$ (resp. $bar_{n+1} (y_1 \ldots, y_{n+1})$. Hence, 
passing to the limit in the previous inequality, we obtain 
$$(n+1) \esp d \big( bar_{n+1} (x_1,\ldots,x_{n+1}), bar_{n+1} (y_1 \ldots, y_{n+1}) \big) 
\leq \sum_{i=1}^{n+1} d(x_i, y_i),$$
as we wanted to show.

\vspace{0.4cm}

\begin{lem} \label{one} 
{\em The diameter of the convex closure of every 
bounded subset of $X$ equals its own diameter.}
\end{lem}

\begin{proof} An explicit inductive description of the convex closure of 
a bounded subset $B$ of $X$ ({\em i.e.} 
the smallest convex subset of $X$ containing $B$) proceeds as follows. Letting $B_0 := B$  
and having defined $B_1,\ldots, B_n$, we let $B_{n+1}$ be the union of all geodesics 
with endpoints in $B_n$. Then $B_n \subset B_{n+1}$, and the closure of the union 
$B_{\infty} := \bigcup_n B_n$ is the convex closure of $B$. Since $B_{\infty}$ 
contains $B$, we have $diam (B_{\infty}) \geq diam(B)$. To show the converse 
inequality, it suffices to show that for all $n \geq 0$, 
\begin{equation}\label{to-check}
diam(B_{n+1}) \leq diam(B_n).
\end{equation} 
To check this, given arbitrary points $x,y$ in $B_{n+1}$, we may find 
$x_0,x_1$ and $y_0,y_1$ in $B_n$ such that $x$ (resp. $y$) lies in the 
geodesic joining $x_0$ and $x_1$ (resp. $y_0$ and $y_1$). The convexity 
of the distance along geodesics shows that
$$d(x,y_0) \leq \max \big\{ d(x_0, y_0), d(x_1, y_0) \big\} \leq diam(B_n),$$ 
$$d(x,y_1) \leq \max \big\{ d(x_0, y_1), d(x_1, y_1) \big\} \leq diam(B_n).$$
Another application of this convexity then shows that 
$$d(x,y) \leq \max \big\{ d(x,y_0), d(x,y_1) \big\} \leq diam(B_n).$$
Since $x,y$ were arbitrary points of $B_{n+1}$, this shows (\ref{to-check}). 
\end{proof}

\vspace{0.25cm}

By the symmetry of the 
construction, for every permutation $\sigma$ of $\{1,\ldots,n\}$, 
$$bar_n (x_1, \ldots, x_n) = bar_n (x_{\sigma(1)}, \ldots, x_{\sigma(n)}).$$
Having this in mind, (\ref{needed-estimate}) implies that  
$$d \big( bar_n (x_1,\ldots,x_n), bar_n (y_1,\ldots,y_n) \big) \leq 
\frac{1}{n} \min_{\sigma \in S_n} \sum_{i=1}^{n} d(x_i, y_{\sigma(i)}).$$
The important observation here is that (by a theorem of Garrett Birkhoff; see 
\cite[Introduction]{villani}) the right-side expression above corresponds 
to the 1-Wasserstein distance between certain probability measures. More 
precisely,  
$$\frac{1}{n} \min_{\sigma \in S_n} \sum_{i=1}^{n} d(x_i, y_{\sigma(i)}) = 
W_1(\mu_1,\mu_2),$$
where $\mu_1 := \frac{1}{n} (\delta_{x_1} + \cdots + \delta_{x_n})$ and 
$\mu_2 := \frac{1}{n} (\delta_{y_1} + \cdots + \delta_{y_n}).$ In order 
to obtain a barycenter map that is 1-Lipschitz for the 1-Wasserstein 
metric, this would motivate to define the barycenter of 
\esp $\frac{1}{n}(\delta_{x_1} + \cdots + \delta_{x_n})$ \esp as 
$bar_n (x_1, \ldots ,x_n)$. However, such a definition is not 
intrinsic. For instance, though the $n$-set $(x_1,x_2,\ldots,x_n)$ and 
the $2n$-set $(x_1,x_1,x_2,x_2,\ldots,x_n,x_n)$ should be identified 
as measures, the points 
$bar_n (x_1,x_2,\ldots,x_n)$ and $bar_{2n} (x_1,x_1,x_2,x_2,\ldots,x_n,x_n)$ do not 
necessarily coincide. As a matter of example, the reader may easily check that for 
$X$ being a tripod of endpoints $x,y,z$ and edges of the same length $\ell$, the 
points $bar_4 (x,x,y,x)$ and $b_8(x,x,x,x,y,y,z,z)$ are different. (The former 
is at distance $\frac{7 \ell}{9}$ from $x$, while the second is at distance 
$\frac{2533 \ell}{3150}$ form the same vertex; see Figure 1.)

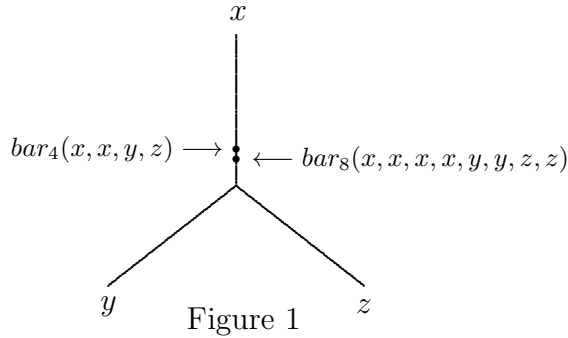
\vspace{0.45cm}

\beginpicture

\setcoordinatesystem units <1cm,1cm>

\plot 0 0 0 2 /
\plot 0 0 -1.7 -1.33 / 
\plot 0 0 1.7 -1.33 /

\put{$x$} at 0.03 2.3
\put{$y$} at -1.7 -1.6 
\put{$z$} at 1.7 -1.6 

\put{Figure 1} at 0.1 -1.8 

\begin{footnotesize}
\put{$bar_4 (x,x,y,z) \longrightarrow$} at -1.7 0.49 
\put{$\longleftarrow bar_8 (x,x,x,x,y,y,z,z)$} at 2.2 0.35  
\end{footnotesize}

\begin{tiny}
\put{$\bullet$} at -0.34 0.49 
\put{$\bullet$} at -0.34 0.35
\end{tiny}

\put{} at -8 0

\endpicture


\vspace{0.53cm}

To solve the problem above, we will slightly modify the definition of the barycenter 
of finite families of points so that it becomes invariant under the procedure 
--at the level of measures-- of ``subdivision of mass along the atoms''. 
Given an arbitrary family $Q = (x_1,\ldots,x_n)$ of points in $X$, we let 
$$Q^k := (x_1,\ldots,x_n,x_1,\ldots,x_n,\ldots,x_1,\ldots,x_n),$$ 
where the number of blocks is $k$. 

\vspace{0.15cm}

\begin{prop} \label{prop-cauchy} {\em The sequence of barycenters 
\esp $bar_{nk} (Q^k)$ \esp is a Cauchy sequence.}
\end{prop}

\vspace{0.15cm}

Assuming that this proposition holds, and since $X$ is 
supposed to be complete, we may define the (canonical) barycenter 
$$bar^{\star} \Big( \frac{1}{n} \big( \delta_{x_1} + \cdots + \delta_{x_n} \big) \Big)$$ 
as the limit point of the sequence \esp $bar_{nk} (Q^k)$. \esp Indeed, one readily 
checks that this limit point depends only on the corresponding measure and not on 
any particular way of writing it as a equally weighted mean of Dirac measures (with 
nonnecessarily different atoms). Moreover, we still have the crucial relation 
$$d \left(bar^{\star} \Big( \frac{1}{n}(\delta_{x_1} + \cdots + \delta_{x_n}) \Big), 
bar^{\star} \Big( \frac{1}{n}(\delta_{x_1} + \cdots + \delta_{x_n}) \Big) \right) 
\leq \frac{1}{n} \min_{\sigma \in S_n} \sum_{i=1}^{n} d(x_i,y_{\sigma(i)}).$$
Thus, denoting by $P_{\mathbb{Q}}(X)$ the 
set of atomic probability measures on $X$ all of whose atoms have rational mass, 
we have a well-defined map \esp $bar^{\star} \!: P_{\mathbb{Q}}(X) \to X$, \esp 
and  the previous inequality translates into that this map is 1-Lipschitz 
for the 1-Wasserstein metric: for all $\mu_1,\mu_2$ in $P_{\mathbb{Q}}(X)$, 
\begin{equation}\label{contract}
d \big( bar^{\star}(\mu_1), bar^{\star}(\mu_2) \big) \leq W_1 (\mu_1,\mu_2).
\end{equation}
If $X$ is separable, then it is known that $P_{\mathbb{Q}}(X)$ is $W_1$-dense 
in $P^1(X)$. We may hence extend the map \esp $bar^{\star}$ \esp to all $P^1(X)$ 
so that (\ref{contract}) holds for all $\mu_1,\mu_2$ in $P^1(X)$. This concludes 
our construction.  

\vspace{0.1cm}

\begin{rem} \label{raro} It is worth pointing out that for CAT(0)-spaces, 
$bar^{\star}$ does not necessarily coincide with the Cartan barycenter. 
Indeed, for the example illustrated by Figure 1, the Cartan barycenter of 
the measure \esp $\frac{\delta_x}{2} + \frac{\delta_y}{4} + \frac{\delta_z}{4}$ 
is the origin, though the barycenter $bar^{\star}$ of this measure lies on the 
axis joining the origin to $x$ (see the final remark of \cite[Section I.2]{EH}).
\end{rem}

\vspace{0.1cm}

To close this section, we next give a proof of Proposition \ref{prop-cauchy}. Let us 
mention that this proposition is also proved in \cite{EH} by means of a quite indirect 
argument that uses a deep martingale theorem and requires $X$ to satisfy a weak 
local-compactness property. Although this very elegant approach does not seem 
to be the most appropriate one in view of the purely geometric nature of the 
statement, the reader will still recognize a certain probabilistic flavor 
in our computations below. The key estimate for the distance between 
the barycenters of $Q^k$ and $Q^{k+l}$ is provided by the next

\vspace{0.1cm}

\begin{lem} \label{two} 
{\em For every \esp $1/2 < \alpha < 2/3$, 
\esp there exists a constant $C = C(\alpha) > 0$ and $L \gg 1$ 
such that for all positive integers \esp $l,k$ \esp 
satisfying $L \leq l \leq \sqrt{k}$, one has 
\begin{equation}\label{dure}
d \big(bar_{nk} (Q^k), bar_{n(k+l)}(Q^{k+l}) \big) 
\leq C D \esp \frac{l^{3 \alpha - 1}}{k},
\end{equation}
where $D$ denotes the diameter of the set \esp $\{x_1,\ldots,x_n\}$. 
Moreover, for $0 \leq l \leq L$, one still has the weaker estimate}
\begin{equation}\label{facil}
d \big(bar_{nk} (Q^k), bar_{n(k+l)}(Q^{k+l}) \big) \leq D \esp \frac{l}{k}.
\end{equation}
\end{lem}

\vspace{0.15cm}

Assuming that this lemma holds, let us prove Proposition \ref{prop-cauchy}. Given 
$\varepsilon > 0$, fix an integer $k_{\varepsilon} \geq \max \{L,10\}$ such that
$$\frac{D}{k_{\varepsilon}} + 
\frac{3^{3-3\alpha}CD}{(2-3\alpha)(k_{\varepsilon}-1)^{2-3\alpha}} 
< \varepsilon,$$
where $C$ is the constant provided by Lemma \ref{two}.  
For any $k_1 < k_2$ larger than $k_{\varepsilon}$, define the sequence $(\ell_j)$ 
by $\ell_1 := k_{\varepsilon}^2$ and $\ell_{j+1} := \ell_j + [\sqrt{\ell_j}]$. 
One readily checks by induction that \esp $\ell_j \geq (k_{\varepsilon}+j)^2 / 9$ 
\esp holds for all $j \geq 1$. Choose $m \geq 1$ such that 
$\ell_{m} < k_2 \leq \ell_{m+1}$. By Lemma \ref{two}, 
$$d \big( bar_{n \ell_j} (Q^{\ell_{j}}), bar_{n \ell_{j+1}} (Q^{\ell_{j+1}}) \big) 
\leq
CD \frac{\big[ \ell_{j+1} - \ell_j \big]^{3\alpha-1}}{\ell_j} 
\leq
\frac{CD}{\ell_{j}^{\frac{3-3\alpha}{2}}}, \quad j = 1,2,\ldots,m-1.$$
Moreover, 
$$d \big( bar_{n \ell_m} (Q^{\ell_{m}}), bar_{n k_2} (Q^{k_2}) \big) 
\leq 
D \frac{[\ell_{m+1} - \ell_m]}{\ell_m}
\leq
\frac{D}{\ell_{m}^{1/2}}.$$
Using the triangle inequality, this yields
\begin{eqnarray*}
d \big( bar_{n k_1} (Q^{k_1}), bar_{n k_2} (Q^{k_2}) \big) 
&\leq& 
\frac{D}{\ell_m^{1/2}} + 
\sum_{j=1}^{m-1} \frac{CD}{\ell_j^{\frac{3-3\alpha}{2}}}\\ 
&\leq& 
\frac{D}{k_{\varepsilon}} + \sum_{j=1}^{\infty} 
\frac{3^{3-3\alpha}CD}{\big( k_{\varepsilon}+j \big)^{3-3\alpha}}\\
&\leq& 
\frac{D}{k_{\varepsilon}} + 
3^{3-3\alpha}CD \int_{k_{\varepsilon}-1}^{\infty} \frac{dx}{x^{3-3\alpha}}\\
&\leq& 
\frac{D}{k_{\varepsilon}} + \frac{3^{3-3\alpha}CD}{(2-3\alpha)(k_{\varepsilon}-1)^{2-3\alpha}} 
\hspace{0.15cm} < \hspace{0.15cm} \varepsilon,
\end{eqnarray*}
thus showing the Cauchy property.

\vspace{0.45cm}

It remains to prove Lemma \ref{two}. The starting remark is given by the next 

\vspace{0.15cm}

\begin{lem} \label{three} 
{\em Given integers $1 \leq l < m$ and points $x,y_1,\ldots,y_m$ 
in $X$, the distance between $x$ and $bar_m (y_1,\ldots,y_m)$ is smaller than 
or equal to the mean distance between $x$ and the points of the form \esp 
$bar_{m-l}(y_1,\ldots,\hat{y}_{i_1},\ldots,\hat{y}_{i_l},\ldots, y_m)$, 
\esp where $i_1,\ldots,i_l$ range over all possible choices of 
different values in $\{1,\ldots,m \}$ (and each weight equals 
$m(m-1) \cdots (m-l+1) = \frac{m!}{(m-l)!}$).}
\end{lem}

\begin{proof} For $l=1$, this follows as an 
application of (\ref{needed-estimate}) to 
$$bar_m (y_1,\ldots,y_m) = 
bar_m \big( bar_{m-1}(y_1,\ldots,\hat{y}_i,\ldots,y_m), \esp i=1, \ldots ,m \big).$$ 
The general case easily follows by an inductive argument using again 
(\ref{needed-estimate}). \end{proof}

\vspace{0.3cm}

The idea of the proof of Lemma \ref{two} consists in viewing the process of 
``reduction of coordinates'' for passing from $Q^{k+l}$ to $Q^k$ 
as a random process, which should imitate a Bernoulli trial for 
large values of $k \gg l$ (this process 
has an hypergeometric multivariate distribution). For each index $j$, 
the final associated error ({\em i.e.} the difference between $l$ and 
the number of deleted entries $x_j$) should be --in mean-- much smaller than 
$ln$. This allows passing from the elementary though useless upper bound \esp 
$\sim \! Dl/k$ \esp for the distance between the barycenters to the 
much better upper bound \esp $\sim \! CDl^{3\alpha - 1} /k$.

\vspace{0.35cm}

\noindent{\em Proof of Lemma \ref{two}.} As explained above, 
estimate (\ref{facil}) follows as a direct application 
of Lemma \ref{three}, so let us concentrate on (\ref{dure}). 
Lemma \ref{three} again implies that the distance from 
$bar_{nk} (Q^k)$ to $bar_{n(k+l)}(Q^{k+l})$ is smaller than or equal 
to the mean of the distance between $bar_{nk} (Q^k)$ and the points 
$bar_{nk}(y_1,\ldots,y_{kn})$, where $(y_1,\ldots,y_{kn})$ ranges over 
all families that coincide with $Q^{k+l}$ except for the deletion of $ln$ 
entries. Among these families, the number of those for which the deleted 
entries correspond to a $x_j$-position a number of times equal to $i_j$ 
(with $i_1 + \cdots + i_n = nl$) is
$${k+l \choose i_1} {k+l \choose i_2} \cdots {k+l \choose i_n}.$$
Moreover, the distance from the barycenter of such a family to $bar(Q^k)$ 
is smaller than or equal to 
$$\frac{D}{kn} \big( |i_1 - l| + |i_2 - l| + \cdots + |i_n - l| \big).$$
By Lemma \ref{three}, this implies that \esp 
$d \big( bar_{nk} (Q^k), bar_{n(k+l)} (Q^{k+l}) \big)$ 
\esp is smaller than or equal to 
\begin{multline*}
\frac{D}{kn} \sum_{i_1 + \cdots + i_n = nl} 
\frac{{k+l \choose i_1} {k+l \choose i_2} \cdots {k+l \choose i_n}}{{n(k+l) \choose nl}} 
\big( |i_1 - l| + |i_2 - l| + \cdots + |i_n - l| \big) \\
= \frac{D}{k} \sum_{i=0}^{nl} 
\frac{{k+l \choose i}{(n-1)(k+l) \choose nl-i}}{{n(k+l) \choose nl}} |i-l| 
\esp \esp = \esp \esp \frac{D}{k} \sum_{i=0}^{nl - l} 
\frac{{k+l \choose l+i}{(n-1)(k+l) \choose nl-l-i}}{{n(k+l) \choose nl}} i 
\esp + \esp \frac{D}{k} \sum_{i=0}^{l} 
\frac{{k+l \choose l-i}{(n-1)(k+l) \choose nl-l+i}}{{n(k+l) \choose nl}} i.    
\end{multline*}
We will estimate the first of the two sums above, leaving to the reader the task 
of carrying out analogous computations for the second sum. First, notice that 
\begin{eqnarray*}
\frac{D}{k} \sum_{i=0}^{nl - l} 
\frac{{k+l \choose l+i}{(n-1)(k+l) \choose nl-l-i}}{{n(k+l) \choose nl}} i 
&=& \frac{D}{k} \sum_{i=0}^{l^{\alpha}} 
\frac{{k+l \choose l+i}{(n-1)(k+l) \choose nl-l-i}}{{n(k+l) \choose nl}} i 
\esp + \esp \frac{D}{k} \sum_{i=l^{\alpha}}^{nl - l} 
\frac{{k+l \choose l+i}{(n-1)(k+l) \choose nl-l-i}}{{n(k+l) \choose nl}} i\\
&\leq& \frac{D l^{\alpha}}{k} \sum_{i=0}^{l^{\alpha}} 
\frac{{k+l \choose l+i}{(n-1)(k+l) \choose nl-l-i}}{{n(k+l) \choose nl}} 
\esp + \esp \frac{D (nl - l)}{k} \sum_{i=l^{\alpha}}^{nl - l} 
\frac{{k+l \choose l+i}{(n-1)(k+l) \choose nl-l-i}}{{n(k+l) \choose nl}}\\
&\leq& \frac{D l^{\alpha}}{k} \esp + \esp \frac{D (n-1) \esp l}{k}
\Big( 1 - \sum_{i=0}^{l^{\alpha}} 
\frac{{k+l \choose l+i}{(n-1)(k+l) \choose nl-l-i}}{{n(k+l) \choose nl}} \Big).
\end{eqnarray*}
The proof will then follow from an estimate of the form
\begin{equation}\label{ultima-reduccion}
1 - \sum_{i=0}^{l^{\alpha}} 
\frac{{k+l \choose l+i}{(n-1)(k+l) \choose nl-l-i}}{{n(k+l) \choose nl}} 
\leq \frac{C}{l^{2-3\alpha}}.
\end{equation}
To show this, first rewrite 
\begin{small}
$$\frac{{k+l \choose l+i}\!{(n-1)(k+l) \choose nl-l-i}}{{n(k+l) \choose nl}} 
\!=\! \frac{{k+l \choose l}\!{(n-1)(k+l) \choose (n-1)l}}{{n(k+l) \choose nl}} 
\frac{k(k-1) \!\cdot\!\cdot\!\cdot\! (k-i+1)}{(l+1) (l+2) 
\!\cdot\!\cdot\!\cdot\! (l+i)} \frac{((n-1)l) ((n-1)l-1) \!\cdot\!\cdot\!\cdot\! 
((n-1)l-i+1)}{((n-1)k) ((n-1)k+1) \!\cdot\!\cdot\!\cdot\! ((n-1)k+i)}.$$ 
\end{small}Now, using the improved version of Stirling's inequality 
(see \cite[Chapter II.9]{feller})
$$\sqrt{2 \pi m} \Big( \frac{m}{e} \Big)^m e^{\frac{1}{12m+1}} \leq 
m! \leq \sqrt{2 \pi m} \Big( \frac{m}{e} \Big)^m e^{\frac{1}{12m}},$$
one easily checks that for a certain \esp 
$e^{\frac{9}{12(l+1)}} \leq \lambda \leq e^{\frac{9}{12 l}}$,  
\begin{equation}\label{choose}
\frac{{k+l \choose l+i}{(n-1)(k+l) \choose (n-1)l-i}}{{n(k+l) \choose nl}} 
= \lambda \esp \sqrt{\frac{(k+l) \esp n}{2 \pi kl (n-1)}} 
\geq \lambda \esp \sqrt{\frac{\esp n}{2 \pi l (n-1)}}.
\end{equation}
On the other hand, choosing $L \gg 1$ and $c > 0$ such that \esp 
$\big| \log(1+x) - x \big| \leq  cx^2$ \esp holds for all 
$|x| \leq 1/L^{2-2\alpha}$, for all $l \geq L$ we have  
\begin{multline*}
\log \left( \frac{k(k-1) \cdots (k-i+1)}{(l+1) (l+2) \cdots (l+i)} 
\frac{((n-1)l) ((n-1)l-1) \cdots ((n-1)l-i+1)}{((n-1)k) ((n-1)k+1) \cdots 
((n-1)k+i)} \right) = \\
= \log \left( \frac{(1-\frac{1}{k}) \cdots (1-\frac{i-1}{k})}
{(1+\frac{1}{l}) (1+\frac{2}{l}) \cdots (1+\frac{i}{l})} 
\frac{(1- \frac{1}{(n-1)l}) \cdots (1-\frac{i-1}{(n-1)l})}
{(1 + \frac{1}{(n-1)k}) \cdots (1 + \frac{i}{(n-1)k})} \right) \\
\geq - \sum_{m=1}^{i-1} \frac{m}{(n-1)l} - \sum_{m=1}^{i} \frac{m}{l} 
- \sum_{m=1}^{i-1} \frac{m}{k} - \sum_{m=1}^{i} \frac{m}{(n-1)k} 
- 2c \Big( \frac{i^3}{l^2} \Big) - 2c \Big( \frac{i^3}{k^2} \Big)\\
\geq - \frac{i^2 n}{2 l (n-1)} - \frac{i(n-2)}{2l(n-1)} - 4cl^{3\alpha-2}. 
\end{multline*}
Putting this together with (\ref{choose}) and using the inequality \esp 
$1-x \leq e^{-x}$, \esp one easily concludes that the expression 
$$\sum_{i=0}^{l^{\alpha}} 
\frac{{k+l \choose l+i}{(n-1)(k+l) \choose nl-l-i}}{{n(k+l) \choose nl}}$$ 
is larger than or equal to 
$$\left( 1 - \frac{\bar{C}}{l^{2-3\alpha}} \right) 
\sum_{i=0}^{l^{\alpha}} \sqrt{\frac{\esp n}{2 \pi l (n-1)}} 
\esp e^{- \frac{i^2 n}{2(n-1)l}}.$$
The involved series can obviously be compared with an integral:
\begin{eqnarray*}
\sum_{i=0}^{l^{\alpha}} \!\sqrt{\frac{\esp n}{2 \pi l (n-1)}} 
e^{- \frac{i^2 n}{2(n-1)l}}
\!\!\!&=&\!\!\! 
\sqrt{\frac{n}{2 \pi (n-1)}} 
\sum_{i=0}^{l^{\alpha}} \frac{e^{- \frac{i^2 n}{2(n-1)l}}}{\sqrt{l}}
\geq
\sqrt{\frac{\esp n}{2 \pi (n-1)}} \int_0^{l^{\alpha}} 
e^{-\frac{x^2 n}{2(n-1)l}} \esp dx\\ 
&\geq&\!\!\!
\frac{1}{\sqrt{2 \pi}} \!\int_{0}^{l^{\alpha} \sqrt{\frac{n}{(n-1)l}}} 
\! e^{-\frac{x^2}{2}} dx
= 1 \!-\! \int_{l^{\alpha}\sqrt{\frac{n}{(n-1)l}}}^{\infty} 
\!\! e^{-\frac{x^2}{2}} dx
\geq
1 - 2 e^{-\frac{\ell^{\alpha}}{2}\sqrt{\frac{n}{(n-1)l}}}.
\end{eqnarray*}
Putting all of this together one easily obtains (\ref{ultima-reduccion}), 
which concludes the proof. $\hfill\square$ 

\vspace{0.535cm}

\noindent{\bf An application: a fixed point theorem.} By construction, 
the map $bar^{\star}$ is equivariant under the action 
of isometries. As a consequence, every action of a compact group by isometries of a 
Buseman space has a fixed point. Indeed, the 
push-forward of the Haar measure along an orbit is an invariant probability measure for the 
action. By equivariance, the barycenter $bar^{\star}$ of this measure must remain fixed.

Despite the simple argument above, it is worth pointing out that 
a much stronger result holds: if a group action by isometries of a 
Buseman space has a (nonempty) compact 
invariant set, then it has a fixed point. (In particular, actions on a proper such space  
with bounded orbits must have fixed points.) Although the author was convinced that 
this was pretty well-known, according to the specialists it is apparently new, so 
we sketch the argument of proof below (the details are left to the reader). 

We will use the following construction. Given a compact subset $B$ of $X$, we let $B^*$ 
be the set of all midpoints between points of $B$ whose distance realizes the diameter. 
By Lemma \ref{one}, 
$$diam(B^*) \leq diam(B) =: D.$$ 
Moreover, if equality holds, then there are points $x_1,x_2,x_3,x_4$ in $B$ such that 
the distance between any of them equals $D$. Indeed, let $y,z$ in $B^*$ be such that 
$d(y,z) = D$. Let $x_1,x_2$ (resp. $x_3,x_4$) be points in $B$ such that $y$ (resp. $z$) 
is the midpoint between $x_1$ and $x_2$ (resp. $x_3$ and $x_4$) and \esp $d(x_1,x_2) = 
d(x_3,x_4) = D$. \esp Using
$$D = d(y,z) \leq \frac{d(x_1,x_3)}{2} + \frac{d(x_2,x_4)}{2} \leq D,$$
we conclude that $d(x_1,x_3) = d(x_2,x_4) = D$. Similarly, using 
$$D = d(y,z) \leq \frac{d(x_1,x_4)}{2} + \frac{d(x_2,x_3)}{2} \leq D,$$
we conclude that $d(x_1,x_4) = d(x_2,x_3) = D$.

The preceding argument easily allows to show the following generalization: starting with 
$B_1 := B$ of diameter $D$, define inductively $B_n := (B_{n-1})^*$. If $diam (B_N) = D$, 
then there exist $2^N$ points $x_1,\ldots, x_{2^N}$ in $B$ such that the distance between 
any of them equals $D$.

Assume now that $\Gamma$ acts on $X$ preserving a compact set $\hat{B}$. Compactness 
type arguments easily yield a compact invariant subset $B$ of $\hat{B}$ of minimal 
diameter $D$. We claim that $B$ is a single point (hence a fixed point for the 
action). Indeed, assume otherwise and cover $B$ by finitely many (say, $M$) 
open balls of radius $D/2$. Since all the $B_n$'s are also compact and 
invariant, the minimality of $D$ yields $diam(B_n) = D$ for all $n \geq 1$. 
Fix $N$ such that $2^N > M$. By the discussion 
above, there exists a sequence of points $x_1,\ldots,x_{2^N}$ in $B$ such that the distance 
between any of them equals $D > 0$. However, this is impossible by the choice of $N$.


\section{The $L^1$ ergodic theorem}
\label{proof}

\hspace{0.5cm} To simplify, given $\varphi \!: \Omega \to X$, let us denote  
$$\mu_{n,\varphi}(\omega) := \frac{1}{m_G (F_n)} \int_{F_n} 
\delta_{\varphi(T^{^g} \omega)} \esp d m_G (g)$$
the $n^{th}$ {\em empirical measure} associated to $\varphi$. Notice that 
for all $\varphi,\psi$ in $L^1(\mathcal{P},X)$ and all $n \geq 1$,
\begin{small}\begin{multline*}
\int_{\Omega} d \Big( bar^{\star} \big(\mu_{n,\varphi} (\omega) \big), 
bar^{\star} \big(\mu_{n,\psi}(\omega) \big) \Big) \esp d \mathcal{P} (\omega)\\ 
=
\int_{\Omega} d \left( bar^{\star} \Big( \frac{1}{m_G (F_n)} \int_{F_n} 
\delta_{\varphi(T^{^g} \omega)} \esp d m_G (g) \Big), 
bar^{\star} \Big( \frac{1}{m_G (F_n)} \int_{F_n} \delta_{\psi(T^g \omega)} 
\esp d m_G (g) \Big) \right) \esp d \mathcal{P} (\omega)\\
\leq
\int_{\Omega} \frac{1}{m_G (F_n)} \int_{F_n} d \big( \varphi(T^g \omega), 
\psi (T^g \omega) \big) \esp d m_G(g) \esp d \mathcal{P} (\omega) \hspace{0.1cm} 
= \hspace{0.1cm} \int_{\Omega} d \big( \varphi (\omega), \psi (\omega) \big) 
\esp d \mathcal{P} (\omega), 
\end{multline*}\end{small}hence
\begin{equation}\label{final}
\int_{\Omega} d \Big( bar^{\star} \big(\mu_{n,\varphi} (\omega) \big), 
bar^{\star} \big(\mu_{n,\varphi}(\omega) \big) \Big) \esp 
d \mathcal{P} (\omega) \esp \leq \esp d_1 (\varphi,\psi).
\end{equation}

\vspace{0.1cm}

To prove the Main Theorem, let us first assume that $\varphi$ takes values in a finite 
set, say $\{x_1,\ldots,x_k\}$, and let $\Omega_i$ be the preimage of $\{x_i\}$ under 
$\varphi$. A direct application of Lindenstrauss' ergodic theorem \cite{lindenstrauss} 
to the characteristic function of $\Omega_i$ yields the existence 
almost everywhere of the following limit:
$$\lambda_i (\omega) := \lim_{n \to \infty} 
\frac{m_G \big( \{ g \in F_n \! : T^g \omega \in \Omega_i \} \big)}{m_G (F_n)}.$$
We claim that almost surely we have the convergence 
\begin{equation}\label{converge}
bar^{\star} (\mu_{n,\varphi}) \longrightarrow 
bar^{\star} \Big( \sum_{i=1}^k \lambda_i (\omega) \delta_{x_i} \Big).
\end{equation}
Indeed, since $bar^{\star}$ is 1-Lipschitz for $W_1$, given $\varepsilon > 0$ 
we have that for almost every $\omega \in \Omega$ there exists 
$n(\omega,\varepsilon) \geq 1$ such that for all 
$n \geq n(\omega,\varepsilon)$ the following holds:
\begin{small}\begin{eqnarray*}
d \left( bar^{\star} (\mu_{n,\varphi}), 
bar^{\star} \Big( \sum_{i=1}^k \lambda_i (\omega) \delta_{x_i} \Big) \right) 
&\leq& W_1 \left( \sum_{i=1}^k 
\frac{m_G \big( \{g \in F_n \!: T^g \omega \in \Omega_i \} \big)}
{m_G (F_n)} \delta_{x_i}, \sum_{i=1}^k \lambda_i(\omega) \delta_{x_i} \right) \\
&\leq& \sum_{i=1}^k \left| \frac{m_G \big( \{g \in F_n \!: T^g \omega \in \Omega_i \}\big)}
{m_G (F_n)} - \lambda_i(\omega) \right| \esp diam \{x_1,\ldots,x_k\}\\
&\leq& \varepsilon. 
\end{eqnarray*}
\end{small}This shows the convergence (\ref{converge}). Now 
notice that by construction, both $bar^{\star} (\mu_{n,\varphi})$ 
and $bar^{\star} \big( \sum_{i=1}^k \lambda_i (\omega) \delta_{x_i} \big)$ belong 
to the convex closure of $\{x_1,\ldots,x_k\}$. By Lemma \ref{one}, this implies 
that for all $n \!\geq\! 1$, the distance between these two points is less than 
or equal to  $diam \{x_1,\ldots,x_k\}$. A direct application of the 
dominated convergence theorem then shows that the convergence 
(\ref{converge}) also holds in $L^{1}(\mathcal{P},X)$.

\vspace{0.32cm}

In order to deal with the general case we will need the next

\vspace{0.1cm}

\begin{lem} \label{lem-maximal} 
{\em There exists a constant $C > 0$ (depending only on the sequence $(F_n)$) 
such that for all $\varphi,\psi$ in $L^1 (X,\mu)$ and all $\lambda > 0$,}
\begin{equation}\label{maximal}
\mathcal{P} \left[ \omega \in \Omega \!: \esp \sup_{n \geq 1}
d \big( bar^{\star} (\mu_{n,\varphi}(\omega)), bar^{\star} (\mu_{n,\psi}(\omega)) \big) 
\geq \lambda \right] \leq \frac{C}{\lambda} d_1 (\varphi,\psi).
\end{equation}
\end{lem}

\begin{proof} Since $bar^{\star}$ is 1-Lipschitz for 
$W_1$, the set involved in the inequality above is contained in 
\esp $\big\{ \omega \in \Omega \!: \esp \sup_{n \geq 1}
W_1 \big( \mu_{n,\varphi}, \mu_{n,\psi} \big) \geq \lambda \big\}$. 
\esp Now, noticing that the measure 
$$\nu_n := \frac{1}{m_G (F_n)} 
\int_{F_n} \delta_{(\varphi(T^g\omega),\psi(T^g\omega))} \esp dm_G (g)$$
lies in \esp $(\mu_{n,\varphi} | \mu_{n,\psi})$, \esp we obtain 
$$W_1 (\mu_{n,\varphi},\mu_{n,\psi}) \leq \frac{1}{m_G (F_n)} \int_{F_n} 
d \big( \varphi(T^g\omega),\psi(T^g\omega) \big) \esp dm_G (g).$$ 
Thus, the left-side expression of (\ref{maximal}) is smaller 
than or equal to 
$$\mathcal{P} \left[ \omega \in \Omega \!: 
\esp \sup_{n \geq 1} \frac{1}{m_G (F_n)} \int_{F_n} 
d \big( \varphi(T^g\omega),\psi(T^g\omega) \big) \esp dm_G (g) 
\geq \lambda \right].$$
Now, a direct application of Lindenstrauss' maximal ergodic theorem 
(see \cite[Theorem 3.2]{lindenstrauss}) yields the existence of 
a constant $C > 0$ (depending only on $(F_n)$) such that this 
last probability is smaller than or equal to 
$$\frac{C}{\lambda} \int_{\Omega} d(\varphi(\omega),\psi(\omega)) 
\esp d \mathcal{P} (\omega),$$ 
as desired. \end{proof}

\vspace{0.15cm}

We may now proceed to complete the proof of the Main Theorem. Since 
$X$ is assumed to be separable, for each $\varphi \!\in\! L^1 (\mathcal{P},X)$ 
there exists a sequence of finite-valued functions 
$\varphi_k \!: \Omega \to X$ that converges to $\varphi$ in the $L^1$ sense. 
Thus, given $\varepsilon > 0$, we may fix $\psi := \varphi_{k_{\varepsilon}}$ 
such that \esp $d_1(\varphi,\psi) \leq \varepsilon^2$. \esp By (\ref{maximal}),  
$$\mathcal{P} \left[ \omega \in \Omega \!: \sup_{n \geq 1} 
d \Big( bar^{\star} \big( \mu_{n,\varphi}(\omega) \big), 
bar^{\star} \big( \mu_{n,\psi}(\omega) \big) \Big) \geq 
\varepsilon \right] \leq \frac{C}{\varepsilon} d_1 (\varphi,\psi) 
\leq C \varepsilon.$$
Since $bar^{\star}(\mu_{n,\psi})$ is known to converge almost everywhere, 
this inequality implies that on a set of measure at least \esp $1-C\varepsilon$,  
\esp the sequence $\big( bar^{\star} (\mu_{n,\varphi}(\omega)) \big)$ asymptotically 
oscillates by at most $2 \varepsilon$. Since this is true for all $\varepsilon > 0$, 
this shows that \esp $bar^{\star} (\mu_{n,\varphi}(\omega))$ \esp converges almost 
surely. 

Finally, to show the convergence in $L^1 (\Omega,X)$, 
just notice that by (\ref{final}), 
\begin{small}\begin{multline*}
\int_{\Omega} d \Big( bar^{\star} \big(\mu_{n,\varphi} (\omega) \big), 
bar^{\star} \big(\mu_{m,\varphi}(\omega) \big) \Big) \esp d \mathcal{P}(\omega)\\
\leq  
\int_{\Omega} \!\left[ 
d \Big(\! bar^{\star} \big(\mu_{n,\varphi}), bar^{\star} \big(\mu_{n,\varphi_k}) \!\Big) 
\!\!+\!
d \Big(\! bar^{\star} \big(\mu_{n,\varphi_k}), bar^{\star} \big(\mu_{m,\varphi_k}) \!\Big) 
\!\!+\!
d \Big(\! bar^{\star} \big(\mu_{m,\varphi_k}), bar^{\star} \big(\mu_{m,\varphi}) \!\Big) 
\! \right]  d \mathcal{P} (\omega)\\
\leq 2 d_1 (\varphi,\varphi_k) + \int_{\Omega}
d \Big( bar^{\star} \big(\mu_{n,\varphi_k}), bar^{\star} \big(\mu_{m,\varphi_k}) \Big) 
\esp d \mathcal{P} (\omega).
\end{multline*}
\end{small}For a given $\varepsilon > 0$, we may fix $k$ large enough 
so that \esp $d_1 (\varphi,\varphi_k) \leq \varepsilon /3$. \esp Since 
$\big( bar^{\star} (\mu_{n,\varphi_k}) \big)$ converges in $L^1(\mathcal{P},X)$ 
as $n$ goes to infinite, we may fix $n_{\varepsilon}$ so that for all $n,m$ 
larger than $n_{\varepsilon}$, 
$$\int_{\Omega} d \Big( bar^{\star} \big(\mu_{n,\varphi_k}), 
bar^{\star} \big(\mu_{m,\varphi_k}) \Big) \esp 
d \mathcal{P} (\omega) \leq \frac{\varepsilon}{3}.$$
Putting all of this together we obtain that for all $n,m$ larger than $n_{\varepsilon}$, 
$$\int_{\Omega} d \Big( bar^{\star} \big(\mu_{n,\varphi} (\omega) \big), 
bar^{\star} \big(\mu_{m,\varphi}(\omega) \big) \Big) \esp d \mathcal{P}(\omega) 
\leq \varepsilon.$$
Hence, $\big( bar^{\star} (\varphi_{n,\varphi}) \big)$ is a Cauchy sequence in 
$L^1 (\mathcal{P},X)$, as we wanted to show. 


\begin{footnotesize}


\vspace{0.3cm}

\noindent Andr\'es Navas\\

\noindent Dep. de Matem\'aticas, 
Fac. de Ciencia, Univ. de Santiago\\
 
\noindent Alameda 3363, Estaci\'on Central, Santiago, Chile\\

\noindent E-mail address: andres.navas@usach.cl\\

\end{footnotesize}

\end{document}